\documentclass[12pt]{article}
\usepackage{latexsym}
\usepackage{amssymb}
\usepackage{amsmath}
\usepackage{rotating}
\usepackage{multirow}
\usepackage{mathrsfs} 
\usepackage{enumerate}
\usepackage{wasysym}
\usepackage{esint}
\usepackage{rawfonts}
\input{prepictex}
\input{pictex}
\input{postpictex}
\usepackage[OT2,OT1]{fontenc}
\def\cyr{%
\renewcommand\rmdefault{wncyr}%
\renewcommand\sfdefault{wncyss}%
\renewcommand\encodingdefault{OT2}%
\normalfont
\selectfont}
\DeclareMathAlphabet{\zap}{OT1}{pzc}{m}{it}
\DeclareTextFontCommand{\textcyr}{\cyr}
\def\be{\begin{equation}}
\def\ee{\end{equation}}
\def\bea{\begin{eqnarray*}}
\def\eea{\end{eqnarray*}}

\newcommand{\rad}{\text{\cyr   ya}}

\def\CC{\mathbb C}

\newtheorem{main}{Theorem}

\DeclareMathOperator{\Hess}{Hess}

\DeclareMathOperator{\Aut}{Aut}
\DeclareMathOperator{\Iso}{Iso}

\newtheorem{conj}{Conjecture}

\newtheorem{lem}{Lemma}
\newtheorem{prop}{Proposition}

\newenvironment{proof}{\medskip \noindent
{\bf Proof.}}{\hfill \rule{.5em}{1em}
\\}

\newenvironment{rmk}{\mbox{ }\\{\bf  Remark}\mbox{ }}{
\hfill $\Box$\mbox{}\bigskip}
\def\ZZ{{\mathbb Z}}
\def\RR{{\mathbb R}}
\def\CP{{\mathbb C \mathbb P}}
\begin{document}

\title{Bach-Flat K\"ahler Surfaces}

\author{Claude LeBrun\thanks{Supported 
in part by  NSF grant DMS-1510094.}\\Stony
 Brook University}

\date{January 15, 2017}
\maketitle

\hfill  {\em In memoriam Gennadi Henkin.}

 \begin{abstract} A Riemannian metric on a compact $4$-manifold is said to be Bach-flat if it is a critical point for the 
$L^2$-norm of the Weyl curvature. When the Riemannian $4$-manifold in question is a K\"ahler surface,
we  provide a rough classification 
of solutions,  followed by   detailed results regarding each case in the classification. The most mysterious  case
prominently involves $3$-dimensional  CR manifolds. 
 \end{abstract}
 
 \section{Introduction}
 \label{intro}

 On a smooth connected   compact   $4$-manifold $M$, the {\em Weyl functional} 
\begin{equation}
\label{weylfun}
\mathcal{W}(g) := \int_M\| W\|_{g}^2~d\mu_g~,
\end{equation}
 quantifies  the deviation  of  a Riemannian metric
$g$ from  local conformal flatness. 
Here $W$ denotes the {\em Weyl tensor} of $g$, which is the piece of the Riemann curvature of $g$ complementary to the Ricci tensor, while  
 the norm $\|\cdot\|_g$  and the volume form $d\mu_g$ in the integrand
are    those associated with the given metric $g$. 
 The Weyl functional \eqref{weylfun} is invariant   not only under the  action of 
the diffeomorphism group (via pull-backs), but also under the action  of the smooth positive functions ${\zap f}: M\to \RR^+$
 by conformal rescaling $g \rightsquigarrow {\zap f}g$.

It is both natural and  useful to study   metrics that are critical points of Weyl functional. The   Euler-Lagrange equations for  such  a metric can be expressed 
\cite{Bach,bes}  as 
$B=0$,
where the {\em Bach tensor} $B$ is defined by
$$
B_{ab}:=(\nabla^c\nabla^d+\frac{1}{2}r^{cd})
W_{acbd}, 
$$
so  these critical metrics  are  said to be 
{\em Bach-flat}.  For reasons reviewed  in \S \ref{pdq},  every $4$-dimensional 
 {conformally  Einstein} metric  is Bach-flat, as is every  {anti-self-dual} metric.
 Conversely, if  the Bach-flat manifold $(M^4,g)$ also happens to be {\em K\"ahler}, Derdzi\'{n}ski \cite[Prop. 4]{derd} discovered that 
 the  geometry of $g$ must locally be of one of these two types near a generic point. Our purpose here is to 
 sharpen this observation into a global classification of solutions. Our main result is the following:

\begin{main} \label{summary}
Let $(M^4, g, J)$ be a  compact connected  Bach-flat  K\"ahler surface. 
Then $g$ is either anti-self-dual, or else is  conformally Einstein on an open dense subset
of $M$. Moreover, the geometric behavior of $(M^4, g, J)$ fits into  exactly one slot of the following 
classification scheme: 
\begin{enumerate}[{\rm I.}]
\item The scalar curvature satisfies $s > 0$ everywhere. In this case,  there are  just two possibilities:
\begin{enumerate}[{\rm (a)}]
\item $(M,g, J)$ is K\"ahler-Einstein, with $\lambda >  0$; or else 
\item $(M, s^{-2}g)$ is Einstein,  with $\lambda >  0$, but has holonomy $\mathbf{SO}(4)$. 
\end{enumerate} 
\item The scalar curvature satisfies $s \equiv 0$. There are again two possibilities: 
\begin{enumerate}[{\rm (a)}]
\item $(M, g, J)$ is K\"ahler-Einstein, with $\lambda =  0$; or else 
\item  $(M,J)$ is a (possibly blown-up) ruled surface, and $g$ is anti-self-dual, but $M$ is not even homeomorphic to an Einstein manifold. 
\end{enumerate} 
\item The scalar curvature satisfies $s < 0$ somewhere. Then there are again exactly two possibilities: 
\begin{enumerate}[{\rm (a)}]
\item $(M,g, J)$ is K\"ahler-Einstein, with $\lambda < 0$; or else 
\item $(M,J)$ is a (possibly blown-up) ruled surface, and  $s$  vanishes exactly along a smooth connected totally umbilic  hypersurface $\mathcal{Z}^3\subset M^4$. 
Moreover, $M- \mathcal{Z}$ has precisely 
 two connected components,  and on both  of these $h:=s^{-2}g$  is a complete 
Einstein metric with $\lambda < 0$. 
\end{enumerate}
\end{enumerate}
\end{main}

A great deal is already known about most  cases in this classification:

\begin{itemize}
\item 
A complex surface admits \cite{chenlebweb,sunspot,tian} a metric of class {I} iff it has $c_1> 0$. Moreover, this metric is always unique \cite{bandomab,lebuniq} up to 
complex automorphisms and homotheties. 
\begin{itemize}
\item The relevant metric is of type {I(a)} iff the complex automorphism group is reductive. 
\item  Up to isometry and rescaling, there are \cite{lebuniq} exactly two solutions of type {I(b)}. 
\end{itemize}
\item Metrics of class {II} are  generally called  scalar-flat K\"ahler metrics. These are exactly the K\"ahler metrics that are anti-self-dual, in the sense that
the self-dual Weyl tensor $W_+$  vanishes identically. 
\begin{itemize}
\item The metrics of type {II(a)}   are often called Calabi-Yau metrics.  A K\"ahler-type complex surface admits such  metrics  \cite{yauma} iff it has  $c_1=0$ mod torsion. 
 This happens 
  exactly for  the minimal complex surfaces of  Kodaira dimension $0$ for which  $b_1$ is even.   When a complex surface admits such a metric,  there is  then  exactly one 
such metric in every K\"ahler class. 
\item  Any complex surface that  admits a metric of type {II(b)} must be projective-algebraic and have  Kodaira dimension $-\infty$; 
they all violate the Hitchin-Thorpe inequality, and most of them are non-minimal. 
Such solutions exist in 
great abundance \cite{klp,leblown,rolsing}; in particular,   any    complex surface with $b_1$ even and 
Kodaira dimension $-\infty$   has blow-ups  that admit metrics of this type. 
\end{itemize}
\item The two cases  that make up class {III} are wildly different. 
\begin{itemize}
\item A complex surface admits a metric of  type {III(a)} iff \cite{aubin,yau} it has $c_1 < 0$. This happens \cite[Prop. VII.7.1]{bpv} iff the complex surface 
is minimal,  has Kodaira dimension $2$,  and contains no rational curve of  self-intersection $-2$.

\item  By contrast, any complex surface admitting a metric of type {III(b)} must have Kodaira dimension $-\infty$. Infinitely many solutions  of this type are currently known
\cite{hwasim,tofribach}. However,  the known solutions all display  peculiar features that  seem most likely to  just be 
artifacts of  the method of construction. 
\end{itemize}
\end{itemize}

Our exposition begins, in \S \ref{pdq}, with a review of some key background facts that will help set the stage  for our main results. We then develop 
 the basic trichotomy of solutions  in \S \ref{tri}.  The proof of Theorem \ref{summary} is then completed in \S \ref{rapt} by carefully proving various auxiliary 
assertions regarding  specific cases  in our classification scheme. The article then concludes by proving some additional  results
about specific types of solutions, with a particular focus on  interesting open problems.

%
%
 
\section{The Weyl Functional}
\label{pdq}

If $(M,g)$ is any smooth compact oriented Riemannian $4$-manifold, 
the Thom-Hirzebruch signature theorem $\mathbf{p_1} = 3\tau$ implies
a Gauss-Bonnet-like integral formula 
\begin{equation}
\label{th} 
\tau (M) = \frac{1}{12\pi^2}   \int_{M}\left( |W_+|^{2}_g- |W_+|^{2}_g\right) d\mu_{g}
\end{equation}
for the signature $\tau = b_+ - b_-$ of $M$. In this formula,  $W_\pm = (W\pm \star W)/2$  denotes
the  self-dual (respectively, anti-self-dual) 
part of the Weyl curvature 
$${W^{ab}}_{cd} ={\mathcal{R}^{ab}}_{cd} - 2 {r}^{[a}_{[c}\delta^{b]}_{d]}+\frac{s}{3} \delta^a_{[c}\delta^b_{d]}$$
of the given metric $g$, here expressed in terms of the Riemann curvature tensor $\mathcal{R}$, Ricci tensor $r$, and scalar curvature $s$.
We remind the reader of the fundamental fact that ${W^a}_{bcd}$ is conformally invariant;  this  in particular explains
the conformal invariance of the Weyl functional   \eqref{weylfun}. 
But now notice  that \eqref{th}  allows one to re-express the Weyl functional \eqref{weylfun} 
as
\begin{equation}
\label{funky}
{\mathcal W}(g)= 
 -12\pi^2 \tau (M) + 2\int_M | W_+|^2 d\mu ~.
\end{equation}

Now, for  any  smooth 1-parameter family of metrics 
$$g_t:= g +t\dot{g}+ O(t^2)$$ 
the first variation of the Weyl functional is   given by 
	$$\left. \frac{d}{dt}{\mathcal W}(g_{t})\right|_{t=0} =
	-\int {\dot{g}}^{ab} B_{ab} ~d\mu $$
where \cite{Bach,bes} the {\em Bach tensor} $B$ is given by
\begin{equation}
\label{bach1}
B_{ab}=(\nabla^c\nabla^d+\frac{1}{2}r^{cd})
W_{acbd}
\end{equation}
Notice that the contracted Bianchi identity
$$\nabla^aW_{abcd} = \nabla_{[c} r_{d]b} + {\textstyle \frac{1}{6}} g_{b[c} \nabla_{d]}s$$
 implies that any Einstein metric 
satisfies the {\em Bach-flat} condition $B=0$. However, since 
 the conformal invariance of the Weyl functional  also makes it
clear  that  the  Bach-flat condition is 
conformally invariant, it follows that any conformally Einstein $4$-dimensional metric 
 is automatically Bach-flat.

 On the other hand,  any oriented Riemannian 4-manifold satisfies the remarkable identity
$$(\nabla^a\nabla^b + \frac{1}{2}r^{ab}){(\star W)}_{cabd}  =0, $$
which encodes the fact that the first variation of \eqref{th} is zero. 
This allows one to  rewrite \eqref{bach1} as 
 \begin{equation}
\label{bach2}
 B_{ab}= (2\nabla^c\nabla^d+r^{cd})(W_+)_{acbd} . 
\end{equation}
In particular, any metric with $W_+\equiv 0$ is automatically Bach-flat. Indeed, 
equation \eqref{funky} shows that such {\em anti-self-dual} metrics are  actually
minimizers of the Weyl functional, and so must  satisfy the associated Euler-Lagrange equation 
$B=0$. 
  
The Bach tensor is automatically symmetric, trace-free, and divergence-free. This reflects the
fact  that  $-B$ is the  gradient of $\int|W|^2d\mu$,
which  is invariant under diffeomorphisms and rescalings. 
Since  $B$ must therefore be $L^2$-orthogonal to any tensor
field of the form $ug_{ab}$ or $\nabla_{(a} v_{b)}$, we  have
\begin{equation}
\label{harmonica}
B_{ab}= B_{ba}, ~~ {B_a}^a=0, ~~\nabla^aB_{ab}=0
\end{equation}
for any $4$-dimensional Riemannian metric.  

We now narrow our discussion to the case of K\"ahler metrics. 
For any K\"ahler metric $g$ on a complex surface $(M,J)$,   with the orientation 
induced by $J$, the self-dual Weyl tensor
is given by 
\begin{equation}
\label{kahlerweyl}
{(W_+)_{ab}}^{cd}= \frac{s}{12}\left[
\omega_{ab}\omega^{cd} - \delta_{a}^{[c} \delta_{b}^{d]}+ {J_a}^{[c}{J_b}^{d]}
\right]
\end{equation}
and so 
is completely determined by the 
scalar curvature and the K\"ahler form $\omega = g(J \cdot , \cdot)$. 
In particular,  we therefore have 
\begin{equation}
\label{sebastian}
|W_+|^2 = \frac{s^2}{24}, 
\end{equation}
a key fact whose lack of conformal invariance ceases to seem  paradoxical as soon as one   recalls 
that the K\"ahler condition isn't conformally invariant either. 
In conjunction, equations \eqref{funky} and \eqref{sebastian} now tell us that any Bach-flat K\"ahler metric is a critical point of the Calabi functional 
\begin{equation}
\label{calfun}
{\mathcal C} (\omega ) =  \int s^2 d\mu ~,
\end{equation}
considered either as a functional on a fixed K\"ahler class $\Omega =[\omega ]$, or
on the entire space of K\"ahler metrics, with $\Omega$ allowed to vary. 
In particular, a conformally  Einstein, K\"ahler metric $g$ must
be an extremal K\"ahler metric in the sense of Calabi \cite{calabix}.
One of  several equivalent characterizations of an extremal metric is the requirement   that 
 $\xi :=J\nabla s$  be a Killing field of $g$.


Plugging \eqref{kahlerweyl}  into \eqref{bach2}, we now obtain a concrete  formula 
$$B_{ab}= \frac{s}{6}\mathring{r}_{ab} + \frac{1}{4}{J_a}^c{J_b}^d\nabla_c\nabla_d s
+\frac{1}{12} \nabla_a\nabla_b s + \frac{1}{12}g_{ab}\Delta s$$
for the Bach tensor of any K\"ahler metric, where 
$\mathring{r}$ denotes  the trace-free part 
$$\mathring{r}_{ab} = r_{ab} - \frac{s}{4}g_{ab}$$
of the Ricci curvature. 
Setting $J^*(B) = B(J\cdot , J \cdot )$, we next decompose the Bach tensor 
$$B=B^\boxplus +B^\boxminus$$
into its $J$-invariant and $J$-anti-invariant parts, and observe that 
\begin{eqnarray}
\label{johann} 
B^\boxplus &:=& \frac{1}{2}\left[ B + J^*(B)\right]=
\frac{1}{6} \Big[ s \mathring{r} + 2 \Hess^\boxplus_0(s)\Big] \\
B^\boxminus&:=& \frac{1}{2}\left[ B - J^*(B)\right]=\frac{1}{12} \Big[  \Hess (s) -
J^*\Hess (s) \Big] ~.\nonumber 
\end{eqnarray}
Here $\Hess = \nabla\nabla$ denotes  the Hessian of a function, and 
 $\Hess^\boxplus_0$ is its trace-free, $J$-invariant  part. 
Now notice that,  since $s$ is real-valued, $\Hess (s) =
J^*\Hess (s)$ if and only if  
$$\nabla_{\bar{\mu}}\nabla^{\nu} s = g^{{\nu}\bar{\lambda}}\nabla_{\bar{\mu}}\nabla_{\bar{\lambda}}s= 0,$$
and this is exactly Calabi's equation $\overline{\partial}\nabla^{1,0}s =0$ for an extremal K\"ahler metric. Consequently, 
a K\"ahler metric $g$ is extremal iff its Bach tensor $B$ is $J$-invariant. 
When this happens,  conditions  \eqref{harmonica} then tells us that $\psi = B(J \cdot , \cdot )$ is a harmonic 
anti-self-dual $2$-form, and  the symmetric  tensors
$g+ tB$ are therefore   $J$-compatible 
K\"ahler metrics for all small $t$. Since $\dot{g}=B$ for this variation of the metric, and since $-B$ is the gradient of $\mathcal{W}$, 
it therefore follows  \cite{chenlebweb}    that   any critical point of the Calabi functional $\mathcal{C}$ on the space of {all} 
$J$-compatible K\"ahler metrics must   actually be Bach-flat.

Revisiting   \eqref{johann}  now reveals  that a  K\"ahler metric is Bach-flat iff 
it is extremal  and satisfies
\begin{equation}
\label{rico}
0= s\mathring{r} + 2  \Hess_0(s).
\end{equation}
However, 
 the trace-free Ricci tensor $\mathring{r}$ always transforms  \cite{bes} under conformal changes $g\rightsquigarrow u^2g$ 
by 
\begin{equation}
\label{riforma}
\mathring{r}\rightsquigarrow 
\hat{\mathring{r}}= \mathring{r} +2 u \Hess_0 (u^{-1})~.
\end{equation}
Thus, as was first pointed out by Derdzi\'{n}ski \cite{derd}, 
the peculiar  conformal rescaling $h=s^{-2}g$ of a 
Bach-flat K\"ahler metric satisfies $\mathring{r}=0$, and so is locally Einstein, 
on the (possibly empty) open set where  $s\neq 0$. We  will  now begin to systematically explore the  global ramifications of this observation.

\section{The Basic Trichotomy} 
\label{tri}

In this section, we will study the global behavior of the scalar curvature $s$ on a   compact Bach-flat  K\"ahler surface. 
Our approach hinges on a  general  property of   strictly extremal K\"ahler manifolds:

\begin{lem} 
\label{bott} 
Let  $(M^4, g,J)$  be a  compact connected extremal K\"ahler surface whose scalar curvature $s$ is non-constant.  Then 
 $s : M \to \RR$ is  a generalized Morse function in the sense of Bott \cite{bottcrit}. In other words, 
the  locus where  $\nabla s =0$ is a disjoint union $\bigsqcup C_j \subset M$ of compact submanifolds, and the Hessian $\Hess(s):= \nabla \nabla s$ 
is non-degenerate  on the normal bundle  $(TC_j)^\perp$ of each  $C_j$. Moreover,  each  submanifold $C_j$ is  either a single point or 
 a smooth compact connected complex curve.  
\end{lem}

\begin{proof} Since $s$ is non-constant, $(M,g,J)$ is a {\em strictly} extremal K\"ahler manifold, and $\xi= J\nabla s$ is a {\em non-trivial} Killing field. The critical points
of $s$ are exactly the fixed points of the flow $\{ \Phi_t : M\to M~|~ t\in \RR\}$ generated by $\xi$, and since the diffeomorphisms $\Phi_t$
are all isometries of $(M,g)$, every connected component $C_j$ of the fixed point set is \cite{kobfixed} a totally geodesic submanifold. 

At  $p\in C_j$, let $v\in (T_pC_j)^\perp$ be a unit vector normal to $C_j$, and let 
$\gamma : \RR \to M$ be the unit speed geodesic through $p= \gamma (0)\in C_j$ with initial tangent vector $v$. Since $\xi=0$ only at 
$\bigsqcup C_k$,  we then have 
$\xi|_{\gamma (t)}\neq 0$ for all sufficiently 
small $t> 0$. However, since $\xi$ is Killing, $\xi\circ \gamma$ is a Jacobi field along $\gamma$. Because 
$\xi|_{\gamma (0)}=0$, we must therefore have $\nabla_{\gamma^\prime(0)}\xi\neq 0$, as 
Jacobi's equation is a linear second order ODE, and $\xi$ would therefore vanish identically along 
$\gamma$ if the initial value $(\xi|_p , (\nabla_v\xi)|_p)$ of this solution vanished. 
This shows that  that  $v$ cannot belong to 
  the kernel of $v\mapsto (\Hess  s)(v, \cdot) = \omega ( \cdot , \nabla_v\xi)$;  and since $v$ is an arbitrary unit normal vector, it follows that the restriction of the Hessian 
$\nabla \nabla s$ to the normal bundle $(TC_j)^\perp$  must be non-degenerate. This shows  that  $s: M\to \RR$ is  a  Morse-Bott function, as claimed. 

In particular, the tangent space $TC_j$ of any component of the critical locus must  exactly coincide with the kernel of the Hessian $\nabla\nabla s$
at any point. But since $\nabla^{1,0}s$ is a holomorphic vector field, this Hessian must be $J$-invariant. Hence $TC_j= \ker \Hess (s)$ is also $J$-invariant, and $C_j\subset M$
is therefore a complex submanifold. Since $s$ is non-constant by assumption, and since $M$ is assumed to have complex dimension $2$, the components $C_j$
can only have complex dimension $0$ or $1$. Each component $C_j$ of the critical locus is therefore either a single point or a  totally geodesic compact complex curve. 
\end{proof}

This immediately tells us something useful about  the zero set
$$\mathcal{Z}:=  \{ p\in M~|~ s(p) = 0\}$$ 
of the scalar curvature. 

\begin{lem} \label{lemon} Let $(M,g,J)$ be a compact extremal K\"ahler manifold. 
 If $s\not\equiv 0$, then the open subset $M-\mathcal{Z}$ is dense in $M$. 
\end{lem}
\begin{proof} If $s$ were a non-zero constant, $\mathcal{Z}$ would be empty, and there would be nothing to prove. We may thus assume from now on that 
$s$ is non-constant. Lemma \ref{bott} then tells us that $\nabla s$ and $\nabla \nabla s$ can never vanish at the same point.
In particular,  if $p$ is any point where $s(p)=0$,
 Taylor's theorem with remainder  allows us to construct a  short embedded curve $\gamma : (-\varepsilon , \varepsilon) \to M$ through $p= \gamma (0)$ 
 on which $s\circ \gamma$ vanishes only at the origin. Hence  every point of $\mathcal{Z}$ belongs
to the closure of $M-\mathcal{Z}$. This shows that $M-\mathcal{Z}$ is dense, as claimed. 
\end{proof}

We now specialize to the case of {\em Bach-flat} K\"ahler surfaces. Since  equation \eqref{johann} tells us  that  these K\"ahler manifolds
 are in particular extremal,  the  
above Lemmata therefore   automatically apply. However, we have already observed  that the equation  $B=0$ can be expressed as 
$$0= \mathring{r} + 2 s^{-1} \Hess_0 s$$
on the open set  $M-\mathcal{Z}$ defined by  $s\neq 0$, and, by equation \eqref{riforma},  this is exactly equivalent to saying that 
the metric $h:= s^{-2}g$ on $M-\mathcal{Z}$  satisfies $\mathring{r}=0$. 
Since the doubly-contracted  Bianchi identity $2\nabla\cdot r =  \nabla s$  implies that a $4$-manifold 
with $\mathring{r}=0$ 
must have locally constant scalar curvature, 
this means that the function $\kappa$ defined by
\begin{equation}
\label{capital}
\kappa =  -6s \Delta s - 12 |\nabla s|^2 + s^3
\end{equation}
is locally constant on $M-\mathcal{Z}$; indeed,  on this open set 
$$\kappa=s^3 ( 6\Delta + s) s^{-1}$$
exactly represents the scalar curvature of the local Einstein metric $h=s^{-2}g$. 
On the other hand, 
since elliptic regularity  implies  \cite{ls2}
 that any extremal K\"ahler metric 
is smooth with respect to the complex atlas, 
our definition \eqref{capital} of  $\kappa$ certainly guarantees that it is a smooth function on all of $M$. 
These facts now allow us to   deduce the following:

\begin{lem} On any compact connected Bach-flat K\"ahler surface $(M,g,J)$, 
the smooth function $\kappa : M\to \RR$ defined by \eqref{capital} is necessarily constant. 
\end{lem}
\begin{proof} If $s\equiv 0$, equation \eqref{capital} immediately tells us that 
  $\kappa\equiv 0$, and we are done. We may therefore assume henceforth that  $s\not\equiv 0$. Now  notice that 
the smooth $1$-form $d\kappa$  vanishes on the  set $M-\mathcal{Z}$,  since on this set $\kappa$ is locally the scalar curvature of the Einstein metric $h=s^{-2}g$, 
and is therefore locally constant. 
But since $(M-\mathcal{Z} ) \subset M$ is dense  by Lemma \ref{lemon}, it therefore follows that $d\kappa \equiv 0$ by continuity. 
Integration on paths thus  shows  that $\kappa$ is constant, as claimed. 
\end{proof}

The sign of $\kappa$ thus provides a basic trichotomy that will form the basis of our classification of these manifolds. However, the
sign of $\kappa$ also has a direct interpretation in terms of the behavior of the scalar curvature of the given K\"ahler metric:

\begin{lem} 
\label{sign}
On any compact connected Bach-flat K\"ahler surface $(M,g,J)$, 
the minimum value $\min s$ of the scalar curvature  of $g$  has exactly the same sign (positive, negative, or zero) as the constant $\kappa$. 
\end{lem}
\begin{proof}
When the scalar curvature $s$ is constant,  \eqref{capital} says $\kappa = s^3 = (\min s)^3$, so the claim obviously holds. 
Otherwise, $s$ is non-constant, and  Lemma \ref{bott} tells us that $\Hess (s)\neq 0$ at any critical point of $s$. In particular, if
$p\in M$ is a point where $s$ achieves its minimum, $\Delta s :=   - \nabla^a\nabla_a s <  0$ at $p$, since we now know that 
$\Hess s= \nabla\nabla s$ must be positive semi-definite and non-zero at
a minimum.    However, evaluation of  equation \eqref{capital} at $p$ tells us that   $$ \kappa= s(p)~\left[s^2 - 6\Delta s\right](p),$$
since $|\nabla s|^2(p)=0$. Since $[s^2 - 6 \Delta s ](p) > 0$, this shows that 
 $\kappa$ and $s(p) = \min s$ must have the same sign,  and the result therefore follows.  
\end{proof}

Our next result leads to  a complete understanding of   the   $\kappa=0$ case. 

\begin{prop} A  compact connected K\"ahler surface $(M,g,J)$ is Bach-flat  and has $\kappa=0$ if and only if its scalar curvature $s$ vanishes identically. 
\label{sfk}
\end{prop}
\begin{proof} 
If  our K\"ahler surface $(M,g,J)$ 
has $s\equiv 0$, it is anti-self-dual by \eqref{sebastian},
and therefore Bach-flat; the fact that such a manifold has $B=0$ 
is also directly confirmed by  \eqref{johann}. 
Inspection of equation \eqref{capital} now reveals that it also has $\kappa=0$.

Conversely, 
if $(M,g,J)$ is a Bach-flat K\"ahler surface with 
$\kappa =0$, then Lemma \ref{sign}  tells us that  $\min s =0$. Thus,   $\mathcal{Z}= s^{-1}(0)$ is non-empty, and every $p\in \mathcal{Z}$
is a minimum of $s$. 
We will now argue by contradiction, and assume  that 
  $s\not\equiv 0$. This implies    that $s$ is   non-constant, so 
Lemma \ref{bott} now tells us  that $\Hess(s):= \nabla \nabla s$ must be non-zero at any point where $\nabla s=0$. 
But since any point $p\in \mathcal{Z}$ is  a minimum of $s$,  this means  that  $\Hess(s)\neq 0$ at every point of $\mathcal{Z}$. 
However, equation \eqref{rico} tells us  that the trace-free part $\Hess_0(s)$ of the Hessian {\em does} vanish 
along the locus $\mathcal{Z}$ defined by  $s=0$.  Thus, for every point $p\in \mathcal{Z}$, there is
a constant $a=a(p)\neq 0$ such that 
\begin{equation}
\label{ompholos}
\nabla \nabla s = 2a g
\end{equation}
at $p$. Since $p$ is a minimum of $s$, we  must moreover have $a > 0$. Hence every  $p\in\mathcal{Z}$ is a non-degenerate local minimum of $s$,
and it therefore 
 follows that $\mathcal{Z}$ is discrete. Since $M$ is compact,  this then implies that  $\mathcal{Z}$ is  finite. 

Now let $p\in \mathcal{Z}$ be any point where $s$ vanishes, and let $\varrho$ be the Riemannian distance from $p$ in $(M, g)$. Since  \eqref{ompholos}
guarantees that $\nabla \xi = 2a J$ at $p$, 
 the isometry $\Phi_{\pi /2a}$, gotten by flowing along $\xi$ for time $t={\pi}/{2a}$, therefore fixes $p$,  but  reverses the direction of each geodesic through $p$.  
The Taylor expansion of $s$ in geodesic normal coordinates $x^j$ centered at $p$ therefore contains only terms of even order, and \eqref{ompholos} therefore tells us that 
$$s = a  \varrho^2 + O (\varrho^4 ).$$
On the other hand,  we also have
\begin{eqnarray*}
g_{jk} &=& \delta_{jk}+ O(\varrho^2)\\
g_{jk,\ell}&=&O(\varrho)
\end{eqnarray*}
in geodesic normal  coordinates. 
If we now pass to inverted coordinates $\tilde{x}^j= x^j/(a \varrho^2)$ and set $\rad :=\sqrt{\sum_j (\tilde{x}^j)^2}= 1/(a \varrho)$, the metric  $h=s^{-2}g$ thus 
satisfies 
\begin{eqnarray*}
h_{jk} &=& \delta_{jk}+ O(\rad^{-2})\\
h_{jk,\ell}&=&O(\rad^{-3})
\end{eqnarray*}
 so that $(M-\mathcal{Z} , h)$ is {\em asymptotically flat}. However, since $\kappa=0$, the Einstein metric $h$ is actually Ricci-flat.
 This in particular implies   \cite{bartnik} that each end of $(M-\mathcal{Z}, h)$ has mass zero. The positive mass theorem \cite{syaction}
 therefore asserts that 
$(M-\mathcal{Z} , h)$ is isometric to Euclidean  $\RR^4$. In particular, $g$ is conformally flat on $M-\mathcal{Z}$, and so has 
$W_+\equiv 0$ on this open dense set. But  by \eqref{sebastian}, this means that the K\"ahler metric $g$ satisfies $s\equiv 0$ on $M-\mathcal{Z}$.
But $M-\mathcal{Z}$ is by definition precisely the set where $s\neq 0$, so this is a contradiction!  In other words,  $M-\mathcal{Z}$ must actually be empty, 
and  any compact Bach-flat K\"ahler surface with $\kappa =0$ must therefore  have 
$s\equiv 0$, as claimed. 
 \end{proof}

\begin{rmk} The above proof can be recast in a way that avoids using the positive mass theorem. Indeed,  the  volume growth 
of an asymptotically Euclidean Ricci-flat manifold with one end must be exactly Euclidean in the large-radius limit, and would be even larger if there were several ends. 
The Bishop-Gromov inequality \cite{bishop,gromhaus} thus forces the exponential map of any Ricci-flat asymptotically flat manifold to actually be an isometry. This shows 
that the only such manifold is Euclidean space. 

The  final contradiction could also have been rephrased  so as to  emphasize   topology instead of  geometry. 
For example, if $M-\mathcal{Z}$ were diffeomorphic to  $\RR^4$, it  would  only have
one end, so  $\mathcal{Z}$ would necessarily  consist of a single point $p$, and  $M$  would have  to be homeomorphic to $S^4= \RR^4\cup \{ p\}$. 
But this is absurd, because, for example, the K\"ahler class $[\omega]$  of $(M,g,J)$ is a non-zero element of $H^2(M,  \RR)$, while  $H^2(S^4, \RR)=0$. 
Alternatively, one could obtain a contradiction at this same juncture 
 by  emphasizing that $(M,J)$ is by hypothesis a complex surface, whereas  $S^4$ does not  \cite{wusphere} even admit  an almost-complex structure.
\end{rmk}

It now  remains for us  to analyze the two cases $\kappa > 0$ and $\kappa < 0$. The first of these is simpler, and  is quite thoroughly understood. 
 
\begin{prop} 
\label{pop}
If the constant $\kappa$ is positive, the scalar curvature $s$ of the extremal K\"ahler manifold $(M,g,J)$ is  everywhere positive. Consequently,   
$\mathcal{Z}$ is empty, 
and $(M,h)$ is a compact   Einstein $4$-manifold  with positive Einstein constant. 
\end{prop}
\begin{proof} If $\kappa > 0$,  Lemma \ref{sign} tells us that $\min s > 0$, too. Thus $s> 0$ everywhere on $M$. Hence 
$\mathcal{Z} =s^{-1}(0) = \varnothing$, and  $M- \mathcal{Z} = M$. Thus $h=s^{-2}g$ is a globally defined 
Einstein metric on $M$, with positive Einstein constant $\lambda = \kappa /4$. 
\end{proof}

Previous  results \cite{lebuniq} therefore provide a complete classification of $\kappa > 0$  solutions.  We will say more about this classification in \S \ref{conclusion}  below.

\bigskip 

By contrast,  the  $\kappa < 0$ case is distinctly  more complicated:

\begin{prop}
\label{understand}
If  the constant $\kappa$ is negative, and if $(M,g, J)$ is not K\"ahler-Einstein, then $\mathcal{Z}$ is a smooth connected real hypersurface $\mathcal{Z}^3 \subset M^4$, 
and  the  complement $M-\mathcal{Z}$ of this hypersurface consists of exactly two connected components $M_+$ and $M_-$.  The Einstein manifolds $(M_\pm , h:=s^{-2}g)$ are 
both complete, and have
negative Einstein constant. Moreover, these two manifolds are both Poincar\'e-Einstein, with the same conformal infinity $(\mathcal{Z} , [g|_\mathcal{Z} ])$. 
\end{prop}

\begin{proof}
When $\kappa < 0$,   Lemma \ref{sign}  tells us that $\min s < 0$. If $s$ is constant, the constant is thus negative, and $s$ is  nowhere zero. Equation 
 \eqref{rico} thus implies  that the K\"ahler metric $g$ is actually Einstein. 
 
We may thus henceforth assume  that $s$ is non-constant. Lemma \ref{bott} therefore tells us  that 
$s: M \to \RR$ is  a Morse-Bott function. 
In particular, $\Hess s\neq 0$ at any critical point of $s$. 
If $p$ is a point where $s$ attains its minimum, we therefore have $\Delta s = - \nabla^a\nabla_a s< 0$ at $p$. However, since $\kappa < 0$, we also have $s (p ) = \min s < 0$
by Lemma \ref{sign}. Thus $s\Delta s > 0$ at $p$. On the other hand, since  $\nabla s=0$ at any
 minimum, \eqref{capital} tells us that 
$6 s\Delta s = s^3 - \kappa$
at $p$.  Hence  $\min (s^3 -\kappa) = [s^3-\kappa] (p) = [6s\Delta s](p) > 0$, and we therefore have 
 $s^3 - \kappa > 0$ on all of $M$. Equation \eqref{capital} now  tells us that 
 $$s\Delta s + 2|\nabla s|^2 = \frac{s^3- \kappa}{6} > 0$$
everywhere. In particular, we have $s\Delta s > 0$ at every critical point of $s$.

If $q$ is now a point where $s$ attains its maximum, our Morse-Bott argument  predicts that $\Delta s = -\nabla^a\nabla_a s> 0$ at $q$, since $\Hess s$ is negative semi-definite and 
non-zero at $q$. Since we have shown that $s\Delta s > 0$ at every critical point, and so  in particular at  $q$,   it therefore follows that $\max s = s (q) > 0$.

Since $M$ is connected, and since $s$ takes on both positive and negative values, 
the locus $\mathcal{Z}$ defined by $s=0$ must therefore be  non-empty. 
Moreover,  since  $s\Delta s > 0$ at every critical point of $s$, it follows that the locus $\mathcal{Z}$ defined by $s=0$  cannot contain any critical points. 
In other words,  $0$ is a regular value of $s$, and $\mathcal{Z}=s^{-1}(0)$ is a therefore a smooth compact non-empty   real hypersurface
in $M$.

Now since $s: M\to \RR$ is a  Morse-Bott function, 
and since $\Hess s$ is $J$-invariant, any complex-codimension-one component $C_j$ of the critical set 
must be a local maximum or a local minimum of $s$. The  other critical points of $s$ are isolated, and  $\Hess s$ is non-degenerate at these  critical points,
with  index $0$, $2$, or $4$; those of index $0$ are local minima, those of index $4$ are local maxima, and those of
index $2$ are saddle points where the  Hessian is of type $({+}{+}{-}{-})$. 
This  dictates   the manner in which the sub-level sets $M_t := s^{-1} \left( (-\infty , t]\right)$ can change as we increase $t$. 
Indeed, for regular values $t_1 < t_2$,  the sub-level set 
$M_{t_2}$ is obtained  from  from the lower sub-level set $M_{t_1} $ in a manner determined by the critical points  with  $t_1< s < t_2$
and, up to homotopy,  is gotten by adding
\begin{itemize}
\item   a disjoint,  unattached point  for each isolated local minimum; 
\item a disjoint, unattached connected Riemann surface for each non-isolated  local minimum; 
\item a $2$-disk, attached along its $S^1$ boundary, for each saddle point; 
\item a $4$-disk, attached along its $S^3$ boundary, for each isolated local maximum; and 
\item a $2$-disk bundle over a connected Riemann surface, attached along its circle-bundle boundary, for each non-isolated
local maximum. 
\end{itemize}
Since these  operations always 
entail adding a path-connected space along a path-connected boundary,   different path 
components of $M_{t_1}$ always survive as separate path components of $M_{t_2}$. 
It therefore follows that only one 
 of the $C_j$ can be  a  local minimum $s$, 
since two different local minima 
would necessarily end up in different connected components of the connected $4$-manifold $M$.   
In particular,   the set of all local minima of $s$ is actually the set of all {\em global} minima, and must   
 either  be a connected compact
complex curve or just a single point.   Looking at the same picture upside down, 
 we similarly  see that the set of  local maxima of $s$ coincides with the set of global maxima,
 and must  either  be a single point or a  connected compact complex curve. 
 
 The open set $M_-$ defined by $s < 0$ 
 is  the interior of the compact manifold-with-boundary $M_0:= s^{-1}((-\infty , 0])$,
 and is therefore homotopy equivalent to it. On the other hand, since  $M= M_t$ for any 
 $t > \max s$, we see that $M_0$ and $M$ have the same number of connected components.
 Hence $M_-$ must be connected. However, since $-s: M\to \RR$ is also a Morse-Bott function, 
  the same reasoning also applies when we ``turn the picture upside down.'' The set 
 $M_+$ defined by 
 $s > 0$ is  therefore connected, too. 
 Thus $M-\mathcal{Z}= M_-\sqcup M_+$ consists of exactly two connected components, as claimed. 
 
Similar reasoning  allows us to understand how  $\partial M_t$ changes as we vary $t$. In the region $\min s < s < \max s$, 
$s$ is a Morse function in the standard  sense, and only has   critical points  of index $2$. If $t_1$ and $t_2$ are regular values of $s$ with 
$\min s < t_1 < t_2 < \max s$, 
it therefore follows that $\partial M_{t_2}$ is obtained from $\partial M_{t_1}$ by performing surgeries in dimension $2-1=1$.
In other words, every time one passes a critical point, one  just modifies the $3$-manifold $\partial M_{t_1}$   by performing a {\em Dehn surgery}; that is, 
one just removes   a solid torus $S^1 \times D^2$,  and then glues 
it  back  in via  a self-diffeomorphism of its $S^1 \times S^1$ boundary.
 Since $M_t$ is  connected when $t = \min s+ \epsilon$ for sufficiently small
$\epsilon> 0$, and because    Dehn surgery on a connected $3$-manifold always produces another connected $3$-manifold, 
 it follows that $\partial M_t$ is connected for any regular value $t \in (\min s , \max s )$. In particular, since $0$ is a regular value
 of $s: M\to \RR$, 
it follows  that $\mathcal{Z}= \partial M_0$ is connected, as claimed.

Since the  metric $h=s^{-2}g$ on the interior $M_-$ of the compact manifold-with-boundary $M_0$ is obtained by rescaling a smooth metric
$g$ by the inverse-square of a non-negative  function $-s: M_0\to \RR$ which vanishes only at  $\partial M_0= \mathcal{Z}$ and has non-zero 
normal derivative along the boundary, the Einstein manifold 
$(M_-,h)$ is  conformally compact \cite{grahlee,lebhspace,mazco} and hence, in particular,  is complete. By the same reasoning, $(M_+, h)$ is 
also  conformally compact, and therefore  complete. Since, by construction, both of  these manifolds have conformal infinity $(\mathcal{Z},  [g|_\mathcal{Z}])$ and 
Einstein constant $\kappa/4 < 0$, we have thus established all of our claims. 
\end{proof}

\begin{rmk}
A key point in the above result is  that $\max s > 0$ when $\kappa < 0$ and $s$ is non-constant. This can also be proved  in  the following interesting way:

 Since the flow of $\xi=J\nabla s$ preserves both $g$ and 
the scalar curvature $s$ of $g$,  it therefore also preserves the conformally rescaled metric $h=s^{-2}g$  on the open set
$M-\mathcal{Z}$ where $h$ is defined. If  we  had  $\max s < 0$ for a solution with $s$ non-constant,  $\mathcal{Z}$ would be empty, 
and $(M, h)$ would be a compact Einstein manifold with Einstein constant $\kappa /4 < 0$ which supported a   Killing field
$\xi\not\equiv 0$. But any Killing field $\xi$ satisfies the Bochner formula \cite{bochner} 
\begin{equation}
\label{bow}
0 = \frac{1}{2}\Delta |\xi |^2 + |\nabla \xi |^2 - r(\xi , \xi)
\end{equation}
and since $(M,h)$ would have negative Ricci curvature $r=\frac{\kappa}{4}h$, this  immediately leads to a contradiction, as the 
right-hand-side would be strictly positive at a maximum of $|\xi|^2$.  Hence $(M,g)$  must    have $\max s \geq 0$. 
However,  equation \eqref{capital} tells us that $|\nabla s|^2 = - \kappa/12 > 0$ along the locus $\mathcal{Z}$ where $s=0$.
The maximum of $s$  therefore cannot be achieved  at a point where $s=0$, and 
it thus   follows that we must  actually have     $\max s > 0$, as claimed.  
\end{rmk}

\section{The Proof of Theorem \ref{summary}}
\label{rapt}

The trichotomy laid out in the previous section proves most of Theorem \ref{summary}. It remains only to prove the auxiliary claims made  regarding the 
second case of each of our three classes. In doing so, we will make  repeated use of the following observation:

\begin{lem}
\label{lulu}
 If  the compact connected complex surface $(M^4,J)$ admits a strictly extremal K\"ahler metric $g$,
then  $(M,J)$ is ruled, and    $\tau (M) \leq 0$. 
\end{lem}
\begin{proof} 
By hypothesis,  the scalar curvature $s$ of $g$ is non-constant, and
$\xi = J \nabla s$ is  a non-trivial Killing field. If $p\in M$ is a minimum of the scalar curvature $s$, it is   fixed
by the flow of $\xi$, and 
the  action of $\xi$ is therefore completely determined, via the exponential map of $g$, 
by the induced isometric action on 
$T_pM$ generated   by $\nabla \xi|_p$. 
However, the same observation  allows us to identify the  isotropy subgroup of $p$ in the identity component   $\Iso_0 (M,g)\subset \Aut (M,J)$
of the isometry group 
with a subgroup of $\mathbf{U}(2)$. Since closure of the  $1$-parameter group 
of isometries generated by $\xi$ is a closed connected Abelian subgroup of the istropy group of $p$,
and since $\mathbf{U}(2)$ has rank $2$, this implies  that either $\xi$ is periodic, or that the closure of the  group
it generates is a $2$-torus in $\Iso_0 (M,g)\subset \Aut_0(M,J)$.  

Let us first consider what happens if the isotropy group of $p$ contains a $2$-torus. 
Since the action of this torus on $M$ is modelled, near $p$, on the action of $\mathbf{U}(1) \times \mathbf{U}(1)\subset \mathbf{U}(2)$
on $\CC^2$, the generators give rise to 
two global 
holomorphic vector fields $\Xi_1$ and $\Xi_2$  which both vanish at $p$, but which are linearly independent at generic  nearby points. 
Thus $\Theta =\Xi_1\wedge \Xi_2$ is a holomorphic section  of the anti-canonical bundle $K^{-1}$ 
which vanishes at $p$, but which does not vanish identically. If $\phi$ is a holomorphic section of $K^{\ell}$, 
$\ell > 0$,  then the contraction $\langle  \phi , \Theta^{\otimes \ell} \rangle$ 
is  a global holomorphic function 
on $M$, and so must be  constant. However,  since  $\langle  \phi , \Theta^{\otimes \ell} \rangle$  certainly vanishes at $p$, this constant function consequently    vanishes
identically. Since $\Theta^{\otimes \ell}$  spans  the fiber of $K^{-\ell}$ at a generic point, this shows
that the holomorphic differential $\phi$ vanishes on an open  set, and therefore vanishes identically. Thus, the plurigenera $p_\ell(M) = h^0(M, \mathcal{O} (K^\ell))$, $\ell > 0$,  must  all vanish,
and  the  Kodaira dimension of $(M,J)$ must be $-\infty$. 

On the other hand,  if $\xi=J\nabla s$ is instead periodic,  then $(M,J)$ contains a family of  rational curves. Indeed, if $\xi$ has period $\lambda$, then 
$2\nabla^{1,0}s = \nabla s - i\xi$  generates a holomorphic action of $\CC/\langle i \lambda \rangle$, which we now identify
with the punctured complex plane $\CC^\times$ via $\zeta \mapsto \exp ( 2\pi \zeta/\lambda)$. We will next show that the generic orbit of the resulting 
$\CC^\times$-action on $(M,J)$ can be compactified  by adding two fixed points, corresponding to  $0$ and $\infty$, to produce a rational curve in $M$.  
Indeed,  we already saw in the proof of Proposition \ref{understand}
that the only critical points of $s$ that are not global maxima or minima are  saddle points where $\Hess s$ is of type 
$({+}{+}{-}{-})$. Since  only a $1$-real-parameter family of trajectories of $\nabla s$ ascends to such a saddle point, the 
generic trajectory of $\nabla s$  misses every saddle point, and therefore flows from the minimum value to the maximum value of $s$. 
Now recall that $\min s$ is either attained at a unique point $p_-$, or is attained along a single connected totally geodesic Riemann surface, which we will 
call $\Sigma_-$; similarly, $\max s$ is either attained at a unique point $p_+$, or is attained along a single connected totally geodesic Riemann surface
$\Sigma_+$. The behavior of the flow lines of $\nabla s$ near $\Sigma_\pm$ is particularly simple; asymptotically, they simply approach $\Sigma_\pm$ orthogonally,
at a unique point, since in exponential coordinates $\xi$ just generates rotation in $\CC^2= \{ (z,w)\}$ about the $z$-axis, and since 
$\nabla J=0$, this implies that $\nabla^{1,0}s$  is  therefore given in these coordinates by a constant times $w\frac{\partial}{\partial w} + O(|(z,w)|^3)$. 
On the other hand, near $p_\pm$, the periodicity of $\xi$ analogously allows one to express $\nabla^{1,0}s$ as a constant times
$k z\frac{\partial}{\partial z} + \ell w\frac{\partial}{\partial w} + O(|(z,w)|^3)$ for suitable non-zero integers $k$ and $\ell$. In either case, 
we  not only see  that a  generic $\CC^\times$-orbit may be completed as a holomorphic map $\CP_1\to M$, but also 
that, as we vary the orbit by moving our initial data though a sufficiently small   holomorphic disk $D_\varepsilon$  transverse to a given generic $\CC^\times$-orbit,
the map $D_\varepsilon\times \CC^\times\to M$  induced by the action extends continuously as a map $\Phi:  D_\varepsilon\times\CP_1\to M$; and 
a variant of  the proof of the Riemann removable singularities theorem then shows that this continuous extension $\Phi$ is  necessarily holomorphic.  
In particular, any $\phi\in \Gamma (M, \mathcal{O}(K^{\ell}))$ pulls back as a holomorphic section 
$\Phi^*\phi$ of $\mathcal{O}(K^{\ell})$ on  $D_\varepsilon\times\CP_1$.
But   the restriction of $K^{\ell}_{D_\varepsilon\times\CP_1}$ to any 
$\{ \mbox{point}\} \times \CP_1$
has negative degree $-2\ell$, so it follows that  $\Phi^*\phi\equiv 0$. However, since $\Phi$  is a local biholomorphism  on  $D_\varepsilon\times \CC^\times$, 
it follows that $\phi$ must itself vanish on an open set. Hence $\phi \equiv 0$ 
by uniqueness of analytic continuation. 
Thus, the plurigenera $p_\ell(M) = h^0(M, \mathcal{O} (K^\ell))$, $\ell > 0$,  must  all vanish,
and we thus conclude that   the  Kodaira dimension of $(M,J)$ must be $-\infty$, just as in the previous  case. 

Since $(M,J)$ is of K\"ahler type, 
the Eniques-Kodaira classification \cite{bpv,GH} therefore implies that it is rational or ruled. 
However, since $g$ is a strictly extremal K\"ahler metric on $(M,J)$, the Futaki invariant of $(M,J, [\omega ])$ must also be non-zero \cite{calabix2}.
Since the Futaki invariant $\mathfrak{aut}(M) \to \CC$ kills the derived Lie algebra $[\mathfrak{aut}(M) , \mathfrak{aut}(M)]$,
this  means that  $\Aut_0 (M, J)$  cannot be semi-simple. Consequently,   $(M,J)$ 
cannot be  $\CP_2$. Since $(M,J)$  is rational or ruled, it   is therefore  either a geometrically ruled surface, or a blow-up 
of a geometrically ruled surface. In particular,  $\tau (M) \leq 0$, as claimed. 
\end{proof}

This now allows us to prove a useful fact  regarding  case  {I(b)}: 

\begin{prop} 
\label{holistic}
Let $(M,g,J)$  be a compact connected  Bach-flat K\"ahler surface with $\kappa> 0$. If $g$ is not K\"ahler-Einstein,   then the  conformally related
Einstein metric $h= s^{-2}g$ 
has holonomy $\mathbf{SO}(4)$. 
\end{prop}
\begin{proof} By Proposition  \ref{pop},  $(M^4,g,J)$ must be  a strictly extremal K\"ahler surface, and Lemma \ref{lulu} therefore tells us that 
$(M,J)$ is a ruled surface with $\tau (M) \leq 0$. 
On the other hand, $M$ also 
admits a globally defined Einstein metric  $h=s^{-2}g$  with $\lambda > 0$, so Bochner's theorem \cite{bochner}
implies that $b_1(M)=0$.  
Surface classification \cite{bpv} therefore tells us that  $(M,J)$ is rational, and so in particular is   simply connected.

 Since $M^4$ is simply connected, the holonomy of any metric on $M$ thus coincides with its restricted holonomy, and so 
 is a compact connected Lie subgroup of $\mathbf{SO}(4)$. However, the action of $\mathbf{SO}(4)$ on $\Lambda^2 = \Lambda^+ \oplus \Lambda^-$ gives rise to an isomorphism 
$\mathbf{SO}(4)/\ZZ_2 = \mathbf{SO}(3)\times \mathbf{SO}(3)$, where the two copies of $\mathbf{SO}(3)$ act  on $\Lambda^+$ and $\Lambda^-$, respectively,
via the tautological $3$-dimensional representation. Thus, if the holonomy group were smaller than $\mathbf{SO}(4)$, its image in at least one  factor
$\mathbf{SO}(3)$ would have to be contained in $\mathbf{SO}(2)$, and 
some non-zero self-dual or anti-self-dual
$2$-form  would therefore  be fixed by the holonomy group. Rescaling this $2$-form $\alpha$ to have norm $\sqrt{2}$ with respect to $h$ and then extending it
as a parallel form to all of $M$,  the almost-complex structure $\mathscr{J}$ defined by $\alpha= h(\mathscr{J}\cdot , \cdot )$  would then be 
integrable, and  $(M,h, \mathscr{J})$ would then be a K\"ahler-Einstein manifold. More specifically,  a parallel self-dual $\alpha$ would give rise to a $\mathscr{J}$  compatible
with the same orientation as the original complex structure $J$, while  a parallel  anti-self-dual  $\alpha$ would instead 
give rise to a $\mathscr{J}$  compatible
with the opposite orientation. 

%
%

Now  the Bach-flat K\"ahler metric $g$ is non-Einstein 
by assumption. 
Since $h=s^{-2}g$ is  Einstein  and is conformally related to $g$, it therefore follows that 
$$\int_M s^2_g d\mu_g > \int_M s^2_h d\mu_h  ,$$
as can either be  read off  from the conformal invariance of 
$$\int_M \left( \frac{s^2}{24} - \frac{|\mathring{r}|^2}{2}\right) d\mu = 8\pi^2 \chi (M) - \int_M |W|^2 d\mu $$
or deduced from the more general fact  \cite{obata} that Einstein metrics are always Yamabe minimizers. 
Since the K\"ahler metric $g$ satisfies \eqref{sebastian},  we thus have 
$$\int_M |W_+|_h^2 d\mu_h = \int_M |W_+|_g^2 d\mu_g =  \int_M \frac{s^2_g}{24} d\mu_g > \int_M \frac{s^2_h}{24} d\mu_h. $$
It therefore follows that $h$ cannot satisfy \eqref{sebastian}, and so cannot  be K\"ahler in a manner compatible with the given orientation of $M$. On the other hand, 
equation \eqref{th}  tells us that 
$$\int_M |W_-|_h^2 d\mu_h = -12\pi^2 \tau (M) + \int_M |W_+|_h^2 d\mu_h$$
and since $\tau(M) \leq 0$, we therefore have  
$$\int_M |W_-|_h^2 d\mu_h \geq \int_M |W_+|_h^2 d\mu_h > \int_M \frac{s^2_h}{24} d\mu_h .$$
This shows that $h$ cannot satisfy the reverse-oriented analog of  \eqref{sebastian}, 
and hence  cannot be K\"ahler with respect to a reverse-oriented complex structure, either. 
This proves that, when $s_g$ is non-constant and positive,  the holonomy group of $(M,h)$ is exactly $\mathbf{SO}(4)$, as claimed. 
\end{proof}

Concerning case {II(b)}, we have the following:

\begin{prop}
\label{scalar-flat}
 If $\kappa=0$ and $(M,g,J)$ is not K\"ahler-Einstein, 
then  $(M,J)$ is a (possibly blown-up) ruled complex surface with $c_1^2 < 0$. 
Moreover,  no $4$-manifold homeomorphic to $M$ ever admits an Einstein metric. 
\end{prop}
\begin{proof}
By Proposition \ref{sfk}, a Bach-flat K\"ahler surface $(M,g,J)$ with $\kappa =0$ must have $s\equiv 0$. 
Since one can decompose the the Ricci form $\rho$ of any K\"ahler surface  as 
$$\rho = \frac{s}{4} \omega + \mathring{\rho}$$
where $\mathring{\rho}\in \Lambda^-= \Lambda^{1,1}_\RR \cap (\omega)^\perp$ is the primitive part of $\rho\in \Lambda^{1,1}$, we see  that $\rho$ is anti-self-dual 
whenever $g$ is scalar-flat K\"ahler. It follows \cite{lsd} that any scalar-flat K\"ahler surface satisfies
$$4\pi^2 c_1^2 (M, J) =  \int_M \mathring{\rho}\wedge \mathring{\rho} = -\int_M \mathring{\rho}\wedge \star \mathring{\rho} =-\int_M |\mathring{\rho}|^2 d\mu,$$
so that $c_1^2\leq 0$, with equality if and only if $g$ is Ricci-flat; cf. \cite{laf}.  In particular, if $(M,g,J)$ is not K\"ahler-Einstein, 
we have
$$(2\chi + 3\tau )(M) = c_1^2 (M,J) < 0.$$
However, the Hitchin-Thorpe inequality \cite{bes,gray-ht,hit,tho} tells us that the homotopy invariant $(2\chi + 3\tau )(M)$
must be non-negative for any compact oriented $4$-dimensional Einstein manifold. Thus, if $\kappa =0$ and 
$(M,g,J)$ is not K\"ahler-Einstein, $M$ cannot even be homeomorphic to an Einstein manifold. 

On the other hand, the plurigenera $p_\ell(M) = h^0(M, \mathcal{O} (K^\ell))$, $\ell > 0$, 
 of a non-Ricci-flat scalar-flat K\"ahler surface must all vanish \cite{lsd,yauruled}. Indeed, if
$\phi$  is a holomorphic section of $K^{\ell}$, the Ricci form  satisfies 
\begin{equation}
\label{plurigenus}
\ell \rho = i\partial\bar{\partial}\log |\phi|^2
\end{equation}
away from the zero locus of $\phi$. 
Taking the inner product of both sides with $\omega$  thus yields
$$ 0 = -\ell s = \Delta \log |\phi|^2$$
wherever $\phi\neq 0$. Since $|\phi|^2$ must have a maximum on $M$, the strong maximum principle \
\cite{giltrud} therefore says that $\log |\phi|^2$ is constant away from the zero set of $\phi$.
If $\phi$ does not vanish identically, it therefore  has constant non-zero norm, and 
equation \eqref{plurigenus} then says that $(M,g)$ is  Ricci-flat. 
If $(M,g,J)$ is  not K\"ahler-Einstein, its  plurigenera must therefore all vanish, as claimed. The Kodaira dimension of $(M,J)$ is therefore
$-\infty$, and the Enriques-Kodaira classification \cite{bpv} therefore 
tells us that $(M,J)$ is rational or ruled. Morever, since $c_1^2 < 0$, it certainly can't be $\CP_2$. We thus   conclude that  $(M,J)$  can actually be 
obtained from some ruled surface by blowing up $\geq 0$ points. 
\end{proof}

Putting these facts together, we now immediately have:

\begin{prop} 
\label{planb}
Suppose that $(M,g,J)$ is  a compact connected Bach-flat K\"ahler surface which is not K\"ahler-Einstein. Then $(M,J)$ is ruled.
\end{prop}
\begin{proof} If 
$g$ has $\mathring{r}\not\equiv 0$, 
 equation \eqref{rico}  tells us that the Bach-flat K\"ahler manifold $(M,g,J)$  has either  $s$ non-constant   or  $s\equiv 0$. 
If  $g$ has $s\equiv 0$, then  Proposition \ref{scalar-flat} then tells us that $(M,J)$ is ruled. Otherwise, 
$g$ must be a strictly extremal K\"ahler metric, and Lemma \ref{lulu} then again  guarantees that  $(M,J)$ is ruled, as claimed. 
\end{proof}

Finally, regarding solutions of type {III(b)}, we have the following:

\begin{prop} 
\label{omphalos} 
Let  $(M,g,J)$ be  a compact connected Bach-flat K\"ahler surface with $\kappa < 0$  that  is not K\"ahler-Einstein. 
Then the connected real hypersurface $\mathcal{Z}\subset M$ is totally umbilic with respect to $g$. Moreover, the Weyl curvature
$W$ of $(M,g)$ vanishes identically along $\mathcal{Z}$. 
\end{prop}
\begin{proof} While both of these features are general consequences \cite{fefgra,lebhspace} of the fact that $h=s^{-2}g$ is Poincar\'e-Einstein, 
we will give quick, self-contained proofs that  supply further information in the present special context. 

Let $\check{g}= g|_{\mathcal{Z}}$ denote the induced Riemannian metric (or {\em first fundamental form}) of our hypersurface,  
and let $\gemini$ denote the second fundamental form (or {\em shape tensor}) of $\mathcal{Z}$. To say that $\mathcal{Z}$ is totally 
umbilic just means that $\gemini = {\zap f} \check{g}$ for some function ${\zap f} : \mathcal{Z} \to \RR$, which is  then  just the mean curvature of $\mathcal{Z}$. 
However,   the second fundamental form  can be expressed as    $\gemini = (\nabla \nu)|_{\mathcal{Z}}$
in terms  of any  unit   $1$-form   $\nu$  on $M$ which is  normal to $\mathcal{Z}$ along this hypersurface, and which is 
compatible with the given orientations of $M$ and $\mathcal{Z}$.  
However,  since $\mathcal{Z}$ is defined by $s=0$, and since equation \eqref{capital} tells us that 
$|\nabla s|^2 = - \frac{\kappa}{12}$
along $\mathcal{Z}$, it follows  that $\nu = 2\sqrt{3/|\kappa|} \, ds$ is a valid choice for this  unit  co-normal  field, provided
 we orient $\mathcal{Z}$  in the corresponding manner.  Thus 
$$\gemini = \left. \left( 2\sqrt{\frac{3}{|\kappa|}}  \Hess s \right)\right|_{\mathcal{Z}}$$
gives us a  convenient formula  for the second fundamental form of $\mathcal{Z}$. However,  equation \eqref{johann} 
tells us that 
$
\Hess_0 s =0
$
along the locus  where $s=0$, so 
$\Hess s = - (\Delta s/4)g$
along $\mathcal{Z}$, and hence 
\begin{equation}
\label{omphaloi}
 \gemini  = - \left( \frac{\sqrt{3}}{2\sqrt{|\kappa|}}\Delta s\right) \check{g}.
\end{equation}
This shows that $\mathcal{Z}$ is totally umbilic, as claimed. 

On the other hand, since the K\"ahler metric $g$ has $s=0$ along $\mathcal{Z}$, equation \eqref{sebastian} tells us that $W_+=0$ there, too. 
Thus $W= W_-$ at $\mathcal{Z}$. However, if we let  $\nu$ denote the  unit $1$-form normal to  $\mathcal{Z}$,
there is a natural bundle isomorphism 
\begin{eqnarray*} 
T^*\mathcal{Z}&\longrightarrow& \Lambda^-|_\mathcal{Z}\\
\theta &\longmapsto&(\nu \wedge \theta) - \star (\nu \wedge \theta).
\end{eqnarray*}
Applying the identity 
$$W_- (\varphi - \star \varphi, \varphi - \star \varphi) -  W_+ (\varphi + \star \varphi, \varphi + \star \varphi) = - 4 \mathcal{R} (\varphi , \star \varphi ),$$
to $\varphi = \nu \wedge \theta$, and remembering that $W_+=0$ along $\mathcal{Z}$, 
we therefore see that $W=W_-$ is completely characterized along $\mathcal{Z}$ by knowing    ${{\mathcal{R}^1}_{234}=W^1}_{234}$ for those  
orthonormal co-frame $\{ e^1, \ldots , e^4\}$ in  which $e^1=\nu$. 
But we can easily understand these components, using the fact that $W$ is conformally invariant. Indeed, since a 
conformal change $g\rightsquigarrow u^2 g$ changes the second fundamental form by $\gemini\rightsquigarrow u\gemini + \langle du, \nu\rangle \check{g}$, a hypersurface
is totally umbilic  iff it can be made totally geodesic by  a conformal change. But the unit normal $1$-form  
$\nu$ is parallel along any  totally geodesic hypersurface. In such a conformal gauge, we therefore  have  ${\mathcal{R}^1}_{234} ={W^1}_{234}=0$. Conformal invariance thus  
 allows us to deduce that 
$W$ vanishes identically along $\mathcal{Z}$, as claimed. 
\end{proof}

Theorem \ref{summary} is now an immediate consequence. Indeed, the basic trichotomy into solution classes {I}, {II}, and {III} just reflects the sign of 
$\kappa$, or equivalently, by Lemma \ref{sign}, the sign of $\min s$. 
 Propositions \ref{sfk}, \ref{pop}, and \ref{understand} then explain the basic features of each class. If the solution is K\"ahler-Einstein, 
it is then of type {I(a)}, {II(a)}, or {III(a)}, depending on the sign of $\kappa$. Otherwise, the underlying complex surface is ruled by Proposition \ref{planb},
and the remaining  claims made about solutions of type {I(b)}, {II(b)}, and {III(b)} are then proved by Propositions  \ref{holistic},  \ref{scalar-flat}, 
and \ref{omphalos}, respectively.

\section{Problems and Perspectives} 
\label{conclusion}


\bigskip

As mentioned in \S \ref{intro},  solutions of class {I} have been completely classified \cite{lebuniq}. 
One key fact that enabled this classification was the observation \cite{lebhem}  that  in this case 
$(M,J)$ has  $c_1> 0$,    and  hence is  a Del Pezzo surface \cite{delpezzo}. 
Indeed, since $s\neq 0$,  we can represent $2\pi c_1$ by  $\tilde{\rho}:=\rho + 2i\partial\bar{\partial} \log |s|$, 
where $\rho$ is the Ricci form of $(M,g,J)$, 
and equation \eqref{rico} then tells us that  
$\tilde{\rho}=q(J\cdot, \cdot)$, where the symmetric tensor field  $q$ is given by 
\begin{equation}
\label{rep}
q = \frac{2s+ \kappa s^{-2}}{12} g + s^{-2} |\nabla s|^2 g^\perp .
\end{equation}
Here, $g^\perp$ is defined to be zero on the span of $\nabla s$ and $J\nabla s$, but to coincide with $g$ on the orthogonal complement of this subspace; 
while this of course means that $g^\perp$ is only defined away from the critical points of $s$,  the tensor field  $|\nabla s|^2g^\perp$  has a unique smooth extension
across the critical points, which is explicitly given by declaring it to be zero at this exceptional set. Since $\kappa > 0$ implies that $s> 0$ everywhere, 
it follows that $q > 0$ for a solution of type {I}, and that $(M,J)$ is therefore 
a Del Pezzo surface when $\kappa$ is positive. Conversely, one can  show \cite{chenlebweb,tian,sunspot}
that every Del Pezzo $(M,J)$ admits a $J$-compatible class-{I} solution, and that this solution is moreover unique \cite{bandomab,lebuniq} up to complex automorphism and rescaling. In fact, the solution is K\"ahler-Einstein except in exactly two cases, namely  the blow-up  of $\CP_2$ at one or two distinct points. 
Thus, solutions of type {I(a)} and {I(b)} are distinguished by  whether or not the Lie algebra of holomorphic vector fields on $(M,J)$ is reductive. 

By further elaboration on   this idea, one is led to the following  result: 

\begin{prop}
\label{beth}
Let $g$ and $\tilde{g}$ be be two  $J$-compatible Bach-flat K\"ahler metrics on the same compact complex surface $(M,J)$.
If these two solutions have different types, according to the classification scheme of Theorem \ref{summary}, then one is of type {\rm II(b)}, and the other
is of type {\rm III(b)}. 
\end{prop}
\begin{proof}
Solutions of type {II(a)} and {III(a)} are distinguished from the others by the Kodaira dimension of $(M,J)$. On the other hand, we have just seen that solution of 
class {I} exist precisely on complex surfaces with $c_1 > 0$, and   solutions  of type {I(a)} are then  distinguished from those of type {I(b)} by  whether 
the  Lie algebra of holomorphic vector fields is reductive. 
It thus  only remains  to show that the existence of solution of class {I} precludes the existence of 
a solution of type {II(b)} or {III(b)}. However, any extremal K\"ahler metric on a Del Pezzo surface has positive scalar curvature \cite[Lemmata A.2 and B.2]{lebhem10}.
Consequently, the presence of  a Bach-flat K\"ahler metric of class {I} automatically precludes the existence of a solution of any other class.\end{proof}

This makes the following piece of 
speculation seem irresistible:


\begin{conj} 
\label{wild}  
On a fixed compact complex surface $(M,J)$, any 
 pair of  \linebreak $J$-compatible Bach-flat K\"ahler metrics are necessarily of  the same type, 
in the sense of    Theorem \ref{summary}. 
\end{conj}

Although the currently known solutions of type {III(b)} will almost certainly 
 turn out to be atypical in many respects,   these   known examples  all live on geometrically ruled surfaces,
which necessarily have   $\tau =0$. 
The following result thus explains why Conjecture \ref{wild} is supported by all  known examples: 

\begin{prop}
No compact complex surface $(M,J)$ of signature $\tau=0$ can  admit a pair of $J$-compatible 
Bach-flat K\"ahler metrics that have different types,  in the sense of    Theorem \ref{summary}. 
\end{prop}
\begin{proof}
By Proposition \ref{beth}, it suffices to show that there cannot simultaneously be a solution $g$ of type {II(b)} and a solution $\tilde{g}$ of type {III(b)}. 
However, any solution of class {II} is scalar-flat K\"ahler, and therefore has $W_+=0$ by equation \eqref{sebastian}; and such a solution is
of type {II(b)} iff it is not Ricci-flat.  If we now assume 
that $\tau (M) =0$, equation \eqref{th} then tells us that $W_-=0$, making  $(M,g)$ is conformally flat, as well as scalar-flat. The Weitzenb\"ock formula
for $2$-forms thus simplifies to say that the Hodge Laplacian on $2$-forms coincides with the Bochner Laplacian $\nabla^*\nabla$, and 
 any harmonic $2$-form must therefore be parallel. Since the assumption that $\tau =0$ also  forces $b_-=b_+\neq 0$, 
 our metric    $g$ is K\"ahler with respect to both orientations, and, since $g$ is not flat, it follows  \cite{burbar,lsd} that  $(M,J)$ is therefore a geometrically 
ruled surface of the form $\Sigma\times_\gamma \CP_1$
 for some representation $\gamma: \pi_1(\Sigma)\to \mathbf{PSU}(2)$. The given scalar-flat metric $g$ is then a twisted product metric, 
 obtained by equipping $\Sigma$ and $\CP_1$ with metrics of constant (and opposite) Gauss curvature, and thus   belongs to a $2$-parameter family 
 of constant-scalar-curvature metrics gotten by rescaling $\Sigma$ and $\CP_1$ by arbitrary positive constants. 
 Since $b_2 (M)=2$,  this shows that the Futaki invariant is identically zero on an 
 open set of the K\"ahler cone. However, 
the Futaki invariant is  quite generally a real-analytic function of the K\"ahler class, as can be seen   by locally sweeping out
 the K\"aher cone by real-analytic families of real-analytic K\"ahler metrics. Hence the Futaki invariant of $(M,J)$ vanishes for every K\"ahler class, 
 and it  therefore follows \cite{calabix2} that  any 
 extremal K\"ahler metric on $(M,J)$ must  have constant scalar curvature. Consequently,  $(M,J)$ cannot admit a strictly extremal K\"ahler metric, 
 and no $J$-compatible  solution $\tilde{g}$ of type {III(b)} can  therefore exist on $(M,J)$. 
\end{proof}

There is a more fundamental reason to hope that  Conjecture \ref{wild} might be true. Recall that Bach-flat metrics are critical points of
the Weyl functional, and that Bach-flat K\"ahler metrics are therefore, in particular, critical points of the Calabi energy \eqref{calfun} on the space
of K\"ahler metrics compatible with a given complex structure. However, the known examples are always \cite{simtofri} actually {\em absolute minima} for
the latter problem. 

\begin{conj} 
\label{thing} 
Let $g$ be a Bach-flat K\"ahler metric on a compact complex surface $(M,J)$. Then $g$ is an absolute minimizer of 
the Calabi energy $\mathcal{C}$ on the space of all $J$-compatible K\"ahler metrics on $M$. 
\end{conj}

This conjecture is an easy exercise for solutions of type {I(a)}, {II(a)}, {II(b)}, and {II(a)}. It is moreover also true for solutions of type {I(b)}, 
although the proof \cite{lebuniq} is much more subtle   in this case. By contrast,
we do not currently know whether  Conjecture \ref{thing} holds 
 for {\em general} solutions of this type  {III(b)}, although 
it does in fact hold \cite{simtofri} 
 for all {\em known} solutions.  Notice that Conjecture \ref{thing} would certainly 
imply Conjecture \ref{wild}, since any solution of type {III(b)} necessarily  has $\mathcal{C}> 0$, whereas any solution of type {II(b)} obviously  has
$\mathcal{C}=0$.

\bigskip

These remarks make it obvious that solutions of type {III(b)} represent the area where our understanding of the subject 
remains most deficient. Still, it is not hard to prove a bit more about them. For example:

\begin{prop} Let $(M,g, J)$ be a solution of type {\rm III(b)}.  Then the complete Einstein Hermitan manifold $(M_-, h,J)$ of Proposition \ref{understand} 
has numerically positive canonical line bundle $K_{M_-}$. In other words, every compact holomorphic curve $C\subset M_-\subset M$ satisfies $c_1\cdot C < 0$. 
\end{prop}
\begin{proof} The flow of $-\nabla s= J \xi$ for positive time is holomorphic, and preserves the region $M_-\subset M$ where $s< 0$; moreover,  
 $s$ is  non-increasing under the flow.   However, since 
the holomorphic vector field $\nabla s - i \xi$ has zeroes, its contraction with any holomorphic $1$-form vanishes identically, and
the induced action on the Albanese torus is therefore zero. By duality, the induced action on the Picard torus is also trivial, 
so the action sends any holomorphic curve to a curve to which it is linearly equivalent. In particular, any compact holomorphic curve $C\subset M_-\subset M$ 
gives rise to the same divisor line bundle $L\to M$ as any of its images under the downward flow of the gradient vector field of $s$. Taking the 
limit in $\mathbb{P}[\Gamma (M,\mathcal{O}( L))]$, we can thus represent the limit of the the images of $C$ under the downward flow
by a (typically singular) curve in $M_-$ which is sent to itself by the action of $\CC^\times$. Such a curve is a sum, with non-negative  integer coefficients, 
of the curve $\Sigma_-$ at which the minimum of $s$ is achieved (assuming  the minimum does not occur at an isolated  point) and 
of  ``vertical'' rational curves arising from flow lines descending from  saddle points of $s$ in $M_-\subset M$. It therefore suffices to check
that $c_1$ is negative on $\Sigma_-$ and on any curve tangent to $\nabla s$ and $\xi = J \nabla s$. However, we can 
represent $2\pi c_1$ on $M_-$ by $\tilde{\rho}:=\rho + 2i\partial\bar{\partial} \log |s|$, 
where $\rho$ is the Ricci form of $(M,g,J)$, and the corresponding symmetric tensor field is once again given by \eqref{rep}. 
At critical points of $s$, or in directions tangent to the space of $\nabla s$ and $\xi$, this expression  simplifies 
to just become $(2 s+ \kappa s^{-2}) g /12$, which is negative-definite on the region $M_-$ given by $s < 0$, since 
we also have $\kappa < 0$ for a solution of type {III(b)}. This shows that  the given curve $C$ is homologous to a curve
$\tilde{C}$ on which $c_1\cdot \tilde{C} < 0$, and we therefore have $c_1 \cdot C < 0$, too, as claimed. 
\end{proof}

On the other hand, this says nothing at all about
 $(M_+, J)$, and this is definitely  not a mere  matter of accident. 
For example, when $(M,J)$ is a Hirzebruch surface, 
$M_-$ is a tubular neighborhood of a rational curve on which $c_1$ is negative, while  
$M_+$ is a tubular neighborhood of a curve on which  $c_1$ is positive. Curiously enough, 
$(M_-, h)$ and $(M_+, h)$ are 
 in fact
actually isometric\footnote{These  complete Einstein manifolds   seem to have  been first  discovered  
by B\'erard-Bergery \cite{beber}, who made a systematic study of  cohomogeneity-one Einstein metrics. 
They have subsequently been rediscovered several times by various groups of physicists \cite{chamblin,papo}.} in these examples; however, 
this is not  a  paradox,  because the relevant  isometry is orientation-reversing, and 
so does not intertwine the given complex structures  on $M_\pm$ in any  direct manner.  

\bigskip 

We can also prove some things  about the real hypersurface $\mathcal{Z}\subset M$: 

\begin{prop} Let $(M,g, J)$ be a Bach-flat  K\"ahler surface of type {\rm III(b)}, and let $\mathcal{Z} \subset M$ be the smooth 
 real hypersurface given by $s=0$. Then the compact connected $3$-manifold $\mathcal{Z}$ is Seifert-fibered. Moreover, 
 the restriction of $\xi$ to ${\mathcal{Z}}$ is a non-trivial Killing field of constant length with respect to the induced metric 
  $\check{g}= g|_{\mathcal{Z}}$, and its orbits are therefore geodesics of $\check{g}$. 
 The flow of $\xi$ moreover preserves the  CR structure induced on $\mathcal{Z}$ by $(M,J)$, and, at any   $p\in \mathcal{Z}$, the following are equivalent:
 \begin{itemize}
  \item  the Levi form of the induced CR structure is non-degenerate; 
 \item the Ricci curvature of $\check{g}$ is positive in the direction of $\xi$; 
 \item  the second fundamental form $\gemini$ of $\mathcal{Z}\subset M$ is non-zero; and 
 \item the extrinsic Laplacian $\Delta_g s$ of the scalar curvature is non-zero. 
 \end{itemize}
\end{prop}
\begin{proof}
Equation \eqref{capital} tells us that $|\xi|^2 = |\nabla s|^2 = - \kappa /12 >  0$ along $\mathcal{Z}$, so the restriction of the Killing field $\xi$ to $\mathcal{Z}$
does indeed have constant, non-zero  length. Since  the closure of the group of isometries generated by $\xi$ is a compact connected Abelian 
Lie group, and hence a torus, we can approximate $\xi$ uniformly by non-zero periodic Killing fields, and the choice of such an approximation
 then endows  $\mathcal{Z}$ with 
a circle action for which all isotropy groups are finite, thereby giving it a Seifert-fibered structure. Since 
$\xi$ is a Killing field of constant length with respect to 
$\check{g}$, we also have 
$$\xi^a \nabla_a \xi_b = - \xi^a \nabla_b \xi_a = - \frac{1}{2}  \nabla_b |\xi|^2 =0,$$
on $(\mathcal{Z},\check{g})$, and the trajectories of $\xi$ are therefore geodesic. Finally, since the flow of $\xi$ on $M$ preserves both $J$ and $s$, 
the  flow acts on the hypersurface $s=0$  by CR automorphisms.

Since $s$ is a non-degenerate defining function for $\mathcal{Z}$, the restriction of $i\partial \bar{\partial} s$ to the CR tangent space of $\mathcal{Z}$
exactly represents the Levi form. On the other hand, equation \eqref{rico} tells us that $i\partial \bar{\partial} s$ is a multiple of the K\"ahler form 
$\omega$ along the locus $s=0$, so we therefore conclude that the Levi-form is non-degenerate exactly at there points of $\mathcal{Z}$ where $\Delta s \neq 0$. 
On the other hand,  equation \eqref{omphaloi} tells us  that the  second fundamental form of  $\mathcal{Z}$ is also non-zero exactly at points where 
$\Delta s \neq 0$. Finally, since the unit normal vector field is a constant multiple of $J\xi$, 
the restriction of 
$\gemini ( J\cdot  , \cdot )$ to  $\xi^\perp \subset T\mathcal{Z}$ is a non-zero constant times the intrinsic covariant derivative $\nabla \xi$, and since 
$\mathcal{Z}$ is umbilic, it therefore follows that $\gemini \neq 0$ exactly when $|\nabla \xi|^2$;  but 
the Bochner Weitzenb\"ock formula \eqref{bow} for a Killing field tells us that $|\nabla \xi|^2 \equiv r(\xi , \xi)$ on $(\mathcal{Z},\check{g})$, so the positivity 
of the Ricci-curvature of $\check{g}$ in the direction of $\xi$ is also equivalent to all the other conditions under discussion. 
\end{proof}

In the known examples, $\mathcal{Z}$ is actually  strictly pseudo-convex. Is this a general feature of all solutions, or
is it a mere artifact, reflecting fact that the known solutions have universal covers of cohomogeneity one?

\bigskip

As long as we  only  consider Bach-flat K\"ahler metrics that are compatible with some  fixed complex structure $J$ on $M$, 
Conjecture \ref{wild} claims that 
 the solution type, as per Theorem \ref{summary},  should be completely determined by $(M,J)$. 
 While  there is a preponderance of evidence in favor of such a conjecture, it is also important to notice 
that the type of the solution is certainly not just determined by the diffeotype of $M$ alone. 
For example, while the smooth manifolds $S^2 \times S^2$ and $\CP_2\# \overline{\CP}_2$ each  support
a unique complex structure with $c_1 > 0$,  each also carries an infinite number of other complex structures  realized by the various Hirzebruch surfaces
$\mathbb{P}( \mathcal{O} \oplus \mathcal{O} (\ell))\to \CP_1$, $\ell \geq 0$. Now,  every Hirzebruch surface with $\ell > 2$  carries \cite{hwasim}
a Bach-flat
K\"ahler metric of type {III(b)}, in contrast to the solutions of type {I(b)} that instead exist when $\ell = 0$ and $1$. Similarly, 
the $4$-manifolds arising as $S^2$-bundles over curves of genus $\geq 2$   carry solutions of both of type {II(b)} and {III(b)}; but 
this form of peaceful co-existence  is once again only made made possible by allowing the complex structure to vary.  

Nonetheless, Theorem \ref{summary} does have  consequences that do primarily  reflect the   differential topology of the underlying $4$-manifold:

\begin{prop} 
\label{ez} 
Let $M$ be the underlying smooth $4$-manifold of a  {non-minimal} compact  complex surface of Kodaira dimension $\geq 0$. 
Then there is  no complex structure $J$ on $M$ for which $(M,J)$ admits a  Bach-flat K\"ahler   metric. Moreover, for
complex surfaces of Kodaira dimension $1$, the existence of Bach-flat K\"ahler metrics is similarly  obstructed even when the surface is minimal. 
\end{prop}
\begin{proof} On a compact complex surface  $(M,J)$ of Kodaira dimension $\geq 0$, Theorem \ref{summary} tells us that any Bach-flat K\"ahler metric
must be K\"ahler-Einstein, with Einstein constant $\lambda \leq 0$. This in particular either means that  $c_1^\RR=0$ or  $c_1 < 0$.
Hence  $(M,g)$ must be minimal and have Kodaira dimension $0$ or $2$. However, for complex surfaces of K\"ahler type,  
Seiberg-Witten theory implies  \cite{lky,morgan} that Kodaira dimension is a diffeomorphism invariant,
and  that non-minimality is a moreover a diffeomorphism  invariant whenever the Kodaira dimension is $\geq 0$. 
Thus the operative obstruction really just reflects the differential topology of $M$, in a manner that is  insensitive to the detailed complex geometry 
of the given  $J$. 
\end{proof}

Finally, it should perhaps also be emphasized that 
Proposition \ref{ez}  certainly does not obstruct the existence of  more 
general Bach-flat  metrics. Indeed, a result of  Taubes \cite{tasd} implies that any complex surface has blow-ups that  admit
anti-self-dual metrics. Thus,  there are certainly many non-minimal complex surfaces of each possible Kodaira dimension $\geq 0$
that   do indeed  admit Bach-flat  metrics; it's just  that these metrics don't happen to be conformally  K\"ahler!

\vfill

\noindent 
{\bf Acknowledgments:} This paper is  dedicated to the fond memory of my  friend  Gennadi Henkin,
in the belief that he would have 
been intrigued  by the role  played   by CR manifolds, and would have formulated   perceptive, helpful  questions that could  have
pushed this  investigation even  further. 
I would also like to thank another old friend, Robin Graham, for some useful pointers  regarding  Poincar\'e-Einstein
metrics. Finally, I would like to thank my former student Christina T{\o}nnesen-Friedman for clarifying some important details  of  her beautiful construction of   explicit examples. 
\pagebreak 
%

\end{document}